\documentclass{amsart}
\usepackage[letterpaper,top=3.5cm,bottom=3.5cm,left=2.5cm,right=2.5cm,marginparwidth=1.75cm]{geometry}

\usepackage{mathtools}
\usepackage{amsmath}
\usepackage{amssymb}
\usepackage{amsthm}
\usepackage{latexsym}
\usepackage[square,numbers]{natbib}

\author{Efe Gürel}
\address{TÜBİTAK Natural Sciences High School, Kocaeli, 41400, Turkey}
\email{efegurel54@gmail.com}

\title[Addition Theorems for Jacobi Functions]{Generalized Addition Theorems for the Jacobi Elliptic Functions}


\newtheorem{theorem}{Theorem}[section]

\newtheorem{lemma}[theorem]{Lemma}

\theoremstyle{definition}
\newtheorem{example}[theorem]{Example}

\DeclareMathOperator{\ns}{ns}
\DeclareMathOperator{\nc}{nc}
\DeclareMathOperator{\dn}{dn}
\DeclareMathOperator{\ds}{ds}
\DeclareMathOperator{\cs}{cs}
\DeclareMathOperator{\sn}{sn}
\DeclareMathOperator{\cn}{cn}
\DeclareMathOperator{\nd}{nd}
\DeclareMathOperator{\scJ}{sc}
\DeclareMathOperator{\cd}{cd}
\DeclareMathOperator{\dc}{dc}
\DeclareMathOperator{\sd}{sd}

\DeclareMathOperator{\pq}{pq}
\DeclareMathOperator{\qp}{qp}
\DeclareMathOperator{\pn}{pn}
\DeclareMathOperator{\qn}{qn}
\DeclareMathOperator{\np}{np}
\DeclareMathOperator{\qr}{qr}
\DeclareMathOperator{\rqJ}{rq}

\DeclareMathOperator{\p}{p}
\DeclareMathOperator{\q}{q}
\DeclareMathOperator{\s}{s}
\DeclareMathOperator{\cJ}{c}
\DeclareMathOperator{\dJ}{d}
\DeclareMathOperator{\n}{n}

\begin{document}

\begin{abstract}
    In this paper, we generalize the classical addition theorems for the Jacobi elliptic functions, giving numerous implicit and explicit formulas. We first present determinant identities involving several of the Jacobi functions. Then, we prove explicit addition formulas for a single function. Novel 4-,5-,6- and 7-term addition formulas are also given alongside a general $n$-term addition formula. 
\end{abstract}
\maketitle
\section{Introduction}

Let $\tau$ be in the complex upper half-plane and $\vartheta_1(u,\tau),\vartheta_2(u,\tau),\vartheta_3(u,\tau),\vartheta_4(u,\tau)$ be the Jacobi theta functions with the values $\vartheta_j(0,\tau)=\vartheta_j$ for $j=1,2,3,4$ as in \cite{Whittakter}. We shall suppress the variable $\tau$ in our notation. The elliptic modulus is given by $k=\vartheta_2^2/\vartheta_3^2$ with the elliptic integrals $K=\pi \vartheta_3^2/2$ and $K'=-i\tau K$. The Jacobi elliptic functions are defined as
\begin{align*}
    \sn u&=\frac{\vartheta_3}{\vartheta_2}\frac{\vartheta_1\left( \frac{\pi}{2K}u \right)}{\vartheta_4\left( \frac{\pi}{2K}u \right)},\\
    \cn u&=\frac{\vartheta_4}{\vartheta_2}\frac{\vartheta_2\left( \frac{\pi}{2K}u \right)}{\vartheta_4\left( \frac{\pi}{2K}u \right)},\\
    \dn u&=\frac{\vartheta_4}{\vartheta_3}\frac{\vartheta_3\left( \frac{\pi}{2K}u \right)}{\vartheta_4\left( \frac{\pi}{2K}u \right)}.
\end{align*}
We will use the Glaisher's quotient notation $\pq=\pn/\qn$ with $\text{nn}=1$. The famous addition formulas
\begin{align*}
    \sn(u+v)&=\frac{\sn u\cn v\dn v+\sn v \cn u \dn u}{1-k^2 \sn^2u \sn^2 v},\\
\cn(u+v)&=\frac{\cn u\cn v-\sn u\dn u \sn v \dn v}{1-k^2 \sn^2u \sn^2 v},\\
\dn(u+v)&=\frac{\dn u \dn v-k^2\sn u \cn u \sn v \cn v}{1-k^2 \sn^2u \sn^2 v},\\
\end{align*}
can be derived from the addition formulas of theta functions. These give us the group law of the elliptic curve $y^2=\left(1-x^2\right)\left(1-k^2x^2\right)$ parametrized by $x=\sn u, y=\sn'u=\cn u \dn u$. Indeed, the sum of two points $P_1=(x_1,y_1)$ and $P_2=(x_2,y_2)$ is obtained as \cite{dlmf}
\begin{align*}
    P_1+P_2=\left( \frac{x_1y_2+x_2y_1}{1-k^2x_1^2x_2^2} ,\frac{y_1y_2-\left(k^2+1\right)x_1x_2+2k^2x_1^3x_2}{1-k^2x_1^2x_2^2}+\frac{2k^2x_1y_1x_2^2(x_1y_2+x_2y_1)}{\left( 1-k^2x_1^2x_2^2 \right)^2}\right).
\end{align*}
This is the analogue of the formula governing the group structure of the elliptic curve $y^2=4x^3-g_2x-g_3$ parametrized by the Weierstrass elliptic function $x=\wp(u),y=\wp'(u)$. In this case, we have
\begin{align*}
    \wp(u)+\wp(v)+\wp(u+v)=\frac{1}{4}\left( \frac{\wp'(u)-\wp'(v)}{\wp(u)-\wp(v)} \right)^2
\end{align*}
and
\begin{align*}
    \begin{vmatrix}
            1 & \wp(u) & \wp'(u)  \\
            1 & \wp(v) & \wp'(v)  \\
            1 & \wp(u+v) & -\wp'(u+v) 
\end{vmatrix}=0.
\end{align*}
The last equation is of particular geometric importance, as it states that three points on an elliptic curve that sum to zero are collinear. Cayley \cite{CayleyDetAdd} has proved a similar formula for Jacobi elliptic functions. Cayley's addition theorem reads
\begin{align*}
    \begin{vmatrix}
1 & \sn u_1 & \cn u_1 & \dn u_1 \\
1 & \sn u_2 & \cn u_2 & \dn u_2 \\
1 & \sn u_3 & \cn u_3 & \dn u_3 \\
1 & \sn u_4 & \cn u_4 & \dn u_4 
\end{vmatrix}=0
\end{align*}
under the assumption $u_1+u_2+u_3+u_4=0$. This is the group law of the elliptic curve $y^2=1-x^2, z^2=1-k^2x^2$ parametrized by $x=\sn u, y=\cn u, z=\dn u$. A similar formula due to Forsyth \cite{dlmf} states that
\begin{align*}
    \begin{vmatrix}
\sn u_1 \cn u_1 & \cn u_1 \dn u_1 & \cn u_1 & \dn u_1 \\
\sn u_2 \cn u_2 & \cn u_2 \dn u_2 & \cn u_2 & \dn u_2 \\
\sn u_3 \cn u_3 & \cn u_3 \dn u_3 & \cn u_3 & \dn u_3 \\
\sn u_4 \cn u_4 & \cn u_4 \dn u_4 & \cn u_4 & \dn u_4 
\end{vmatrix}=0
\end{align*}
where $u_1+u_2+u_3+u_4=2K$. The formulas above are derived through Abel's method. They are obtained by intersecting the elliptic curve with equations of codimension $1$. We refer the reader to \cite{Forsyththeory} for these examples and an account which is applicable to curves defined by one or two equations, which is referred to as Abel's theorem in its most general form.\medskip

The addition formulas for $u_1+u_2+u_3$ are particularly interesting for Jacobi functions because of their connection with the fundamental formulas of theta functions. Let us employ the convenience of notation $s_h=\sn u_h, c_j=\cn u_h, d_h=\dn u_h$ for $h=1,2,3$. Glaisher's addition theorem \cite{Glaisher} states
\begin{align*}
    \sn(u_1+u_2+u_3)=\frac{A}{D}, \qquad \cn(u_1+u_2+u_3)=\frac{B}{D},\qquad \dn(u_1+u_2+u_3)=\frac{C}{D}
\end{align*}
where
\begin{align*}
    A&=s_1s_2s_3\left( -1-k^2+2k^2\sum_{cyc} s_1^2 -\left( k^2+k^4 \right)\sum_{cyc} s_2^2s_3^2+2k^4s_1^2s_2^2s_3^2\right)\\
&+\sum_{cyc}^{}\left(s_1c_2c_3d_2d_3\left( -1+2k^2s_2^2s_3^2+2k^2s_1^2-k^2\sum_{cyc}^{}s_2^2s_3^2 \right)  \right),\\
B&=c_1c_2c_3\left( 1-k^2\sum_{cyc}^{}s_2^2s_3^2 +2k^4s_1^2s_2^2s_3^2\right)\\
&+\sum_{cyc}^{}\left(c_1s_2s_3d_2d_3\left( -1+2k^2s_2^2s_3^2+2k^2s_1^2-k^2\sum_{cyc}^{}s_2^2s_3^2 \right)  \right),\\
C&=d_1d_2d_3\left( 1-k^2\sum_{cyc}^{}s_2^2s_3^2 +2k^4s_1^2s_2^2s_3^2\right)\\
&+k^2\sum_{cyc}^{}\left(d_1s_2s_3c_2c_3\left( -1+2k^2s_2^2s_3^2+2s_1^2-k^2\sum_{cyc}^{}s_2^2s_3^2 \right)  \right),\\
D&=1-2k^2\sum_{cyc}s_2^2s_3^2+4\left( k^2+k^4 \right)s_1^2s_2^2s_3^2-2k^4s_1^2s_2^2s_3^2\sum_{cyc}s_1^2+k^4\sum_{cyc}^{}s_2^4s_3^4
\end{align*}
and the summations are cyclic with respect to the indices $1,2,3$. Similarly, we have Cayley's addition theorem \cite{Cayley3}
\begin{align*}
        \sn(u_1+u_2+u_3)=\frac{A'}{D'}, \qquad \cn(u_1+u_2+u_3)=\frac{B'}{D'},\qquad \dn(u_1+u_2+u_3)=\frac{C'}{D'}
\end{align*}
where
\begin{align*}
    A'&=\sum_{cyc}^{}s_1c_2c_3d_2d_3-s_1s_2s_3\left( 1+k^2-k^2\sum_{cyc}^{}s_1^2+k^4s_1^2s_2^2s_3^2 \right),\\
B'&=c_1c_2c_3\left(1- k^4s_1^2s_2^2s_3^2\right)-d_1d_2d_3\sum_{cyc}^{}s_2s_3c_1d_1,\\
C'&=d_1d_2d_3\left(1- k^4s_1^2s_2^2s_3^2\right)-k^2c_1c_2c_3\sum_{cyc}^{}s_2s_3c_1d_1,\\
D'&=1-k^2\sum_{cyc}^{}s_2^2s_3^2+\left( k^2+k^4 \right)s_1^2s_2^2s_3^2-k^2s_1s_2s_3\sum_{cyc}^{}s_1c_2c_3d_2d_3.
\end{align*}
Forsyth \cite{forsythE} has proven the equivalences
\begin{align*}
    &-\frac{k^2s_1s_2s_3s_4}{1+k^2s_1s_2s_3s_4}\sum_{h=1}^4\frac{c_hd_h}{s_h}=\frac{k^2d_1d_2d_3d_4}{k^2+d_1d_2d_3d_4}\sum_{h=1}^{4}\frac{s_hc_h}{d_h}=\frac{k^2c_1c_2c_3c_4}{k^2c_1c_2c_3c_4-1+k^2}\sum_{h=1}^4\frac{s_hd_h}{c_h}\\
=&\frac{k^2s_1s_2s_3s_4d_1d_2d_3d_4}{k^2\left(1-k^2\right)s_1s_2s_3s_4-d_1d_2d_3d_4}\sum_{h=1}^4\frac{c_h}{s_hd_h}=-\frac{k^2c_1c_2c_3c_4d_1d_2d_3d_4}{k^2c_1c_2c_3c_4+d_1d_2d_3d_4}\sum_{h=1}^4\frac{s_h}{c_hd_h}\\
=&\frac{k^2s_1s_2s_3s_4+c_1c_2c_3c_4}{c_1c_2c_3c_4+\left( 1-k^2 \right)s_1s_2s_3s_4}\sum_{h=1}^4\frac{d_h}{s_hc_h}
=-k^2\left( \frac{1}{s_1s_2s_3s_4}+\frac{1}{c_1c_2c_3c_4}+\frac{k^2}{d_1d_2d_3d_4} \right)^{-1}\sum_{h=1}^{4}\frac{1}{s_hc_hd_h}
\end{align*}
when $u_1+u_2+u_3+u_4=0$ holds. Some of the various equations that can be obtained with the aforementioned methods are recorded below under the assumption $u_1+u_2+u_3+u_4=0$:
\begin{align*}
    0&=(s_2-s_1)(c_3d_4-c_4d_3)+(s_4-s_3)(c_1d_2-c_2d_1)\\
&=s_4c_1d_2+s_3c_2d_1+s_2c_3d_4+s_1c_4d_3\\
&=s_4c_2d_1+s_3c_1d_2+s_2c_4d_3+s_1c_3d_4\\
&=\left(1-k^2\right)-k^2\left(1-k^2\right)s_1s_2s_3s_4+k^2c_1c_2c_3c_4-d_1d_2d_3d_4\\
&=\left( 1-k^2 \right)\left( s_1s_2c_3c_4-c_1c_2s_3s_4 \right)-d_1d_2+d_3d_4\\
&=\left( 1-k^2 \right)\left( s_1s_2-s_3s_4 \right)+d_1d_2c_3c_4-c_1c_2d_3d_4\\
&=s_1s_2d_3d_4-d_1d_2s_3s_4+c_3c_4-c_1c_2,
\end{align*}
see \cite{forsythadd,Gudermann,Smith}. The multiplication formulas arise from the addition formulas as special cases. Denoting $s=\sn u, c=\cn u$ and $d=\dn u$, we have
\begin{align*}
    \sn(2u)=\frac{2scd}{1-k^2s^4}, \qquad \cn(2u)=\frac{1-2s^2+k^2s^4}{1-k^2s^4}, \qquad \dn(2u)=\frac{1-2k^2s^2+k^2s^4}{1-k^2s^4}
\end{align*}
and
\begin{align*}
    \sn(3u)&=\frac{3s-4\left( 1+k^2 \right)s^3+6k^2s^5 -k^4s^9}{1-6k^2s^4+4k^2\left( 1+k^2 \right)s^6-3k^4s^8},\\
\cn(3u)&=c\frac{1-4s^2+6k^2s^4-4k^4s^6+k^4s^8}{1-6k^2s^4+4k^2\left( 1+k^2 \right)s^6-3k^4s^8},\\
\dn(3u)&=d\frac{1-4k^2s^2+6k^2s^4-4k^2s^6+k^4s^8}{1-6k^2s^4+4k^2\left( 1+k^2 \right)s^6-3k^4s^8}.
\end{align*}

Recently, the determinant and explicit addition formulas for the Weierstrass elliptic function were generalized in \cite{Gürel}. We explain the setup here. Let us denote $\wp^{(-2)}=1$ for convenience. Take distinct numbers $n_1,\ldots, n_m,k_1,\ldots,k_\ell\in \mathbb{N}\cap \{-2\}$ with given complex numbers $\gamma_1,\ldots,\gamma_m$ and $z_1,\ldots,z_\ell$ with $z+z_1+\ldots+z_\ell=0$. Assume that $\ell=\max(n_j,k_j)+1$. The intersection with the equation
\begin{align*}
    \sum_{i=1}^{m}\gamma_i \wp^{(n_i)}(z_j)=\sum_{i=1}^{\ell}\lambda_i\wp^{(k_i)}(z_j)
\end{align*}
for all $j=1,\ldots,\ell$ is considered. We define the elementary symmetric polynomials $S_r(x_1,\cdots x_n)$, $r\le n$  by
\begin{align*}
    S_{r}(x_1,\ldots,x_n)=\sum_{1\le i_1<\cdots< i_r\le n} x_{i_1} \ldots x_{i_r}
\end{align*}
and denote
\begin{align*}
    S_{r}f(x_1,\ldots,x_n)=S_r\left( f(x_1),\ldots,f(x_n) \right)
\end{align*}
for a function $f$. The main result is as follows.
\begin{theorem}\label{gürelthm}
    The following equation holds,
    \begin{align*}
        \begin{vmatrix}
        \sum_{i=1}^{m}\gamma_i \wp^{(n_i)}(z_1)& \wp^{(k_1)}(z_1) &\cdots &\wp^{(k_\ell)}(z_1) \\
        \sum_{i=1}^{m}\gamma_i \wp^{(n_i)}(z_2)& \wp^{(k_1)}(z_2)  &\cdots &\wp^{(k_\ell)}(z_2)\\
        \vdots& \vdots  & \vdots &\vdots  \\
        \sum_{i=1}^{m}\gamma_i \wp^{(n_i)}(z_\ell)& \wp^{(k_1)}(z_\ell)  &\cdots & \wp^{(k_\ell)}(z_\ell) \\
        \sum_{i=1}^{m}\gamma_i \wp^{(n_i)}(z) & \wp^{(k_1)}(z)  &\cdots & \wp^{(k_\ell)}(z) 
    \end{vmatrix}=0.
    \end{align*}
    Furthermore, if  the constants $c_r$ are determined by
    \begin{align*}
        \left( \sum_{2\nmid n_i}\gamma_i\wp^{(n_i)}(u)- \sum_{2\nmid k_i}\lambda_i\wp^{(k_i)}(u)\right)^2-\left( \sum_{2\mid n_i}\gamma_i\wp^{(n_i)}(u)- \sum_{2\mid k_i}\lambda_i\wp^{(k_i)}(u)\right)^2=\sum_{r=0}^{\ell+1}c_r\wp(u)^r,
    \end{align*}
    then we have
    \begin{align*}
        S_{r}\wp(z_1,\ldots,z_\ell,z)=(-1)^r \frac{c_{\ell+1-r}}{c_{\ell+1}}
    \end{align*}
    and
    \begin{align*}
                \wp(z)=\frac{1}{S_{r-1}\wp(z_1,\ldots,z_\ell)}\left((-1)^r\frac{c_{\ell+1-r}}{c_{\ell+1}} -S_{r}\wp(z_1,\ldots,z_\ell) \right).
    \end{align*}
\end{theorem}
In this paper we will derive general addition theorems for the Jacobi elliptic functions through similar methods. Our results may be regarded as geometric addition theorems of elliptic curves defined by many equations. We will rely on Abel's theorem. Our results include general determinant theorems and explicit addition formulas similar to that of Theorem \ref{gürelthm}.

\section{General Determinant Theorems}
\subsection{Notions and the Main Theorem}

We shall first present the fundamental properties of Jacobi elliptic functions. It is well-known that each one of the Jacobi functions $\pq$ satisfies an algebraic differential equation, the most famous one being $\left(\sn'\right)^2=\left(1-\sn^2\right)\left(1-k^2\sn^2\right)$ as previously stated. We put all these equations together in the form
\begin{align*}
    \left(\pq'\right)^2=A_{\pq}\pq^4+B_{\pq}\pq^2+C_{\pq}
\end{align*}
for constants $A_{\pq},B_{\pq},C_{\pq}$. Explicit forms of these equations are not needed in our investigations and their mere existence will suffice. The reader may consult classic treatises such as \cite{Whittakter} or \cite{dlmf} for these equations. The half-period translations induce a sign change as follows,
\begin{align*}
    \sn(u+2mK+2niK')&=(-1)^{m}\sn u,\\
    \cn(u+2mK+2niK')&=(-1)^{m+n}\cn u,\\
    \dn(u+2mK+2niK')&=(-1)^{n}\dn u.
\end{align*}
Furthermore, the quarter-period translations transform the Jacobi functions into one another. The following table may be found in \cite{dlmf}.
\begin{center}
\begin{table}[h]
    \begin{tabular}{|c|c|c|c|}
    \hline
    Quarter-period $\omega$ & $K$ & $iK'$ & $K+iK'$\\ \hline
    $\sn(u+\omega)$ & $\cd u$ & $\frac{1}{k}\ns u$ & $\frac{1}{k}\dc u$ \\ \hline
    $\cn(u+\omega)$ & $-k' \sd u$ & $-i\frac{1}{k}\ds u$ & -i$\frac{k'}{k}\nc u$ \\ \hline      
    $\dn(u+\omega)$ & $k'\nd  u$ & $-i\cs u$ & $i k'\scJ u$ \\ \hline      
    \end{tabular}\medskip
    
    \caption{Quarter-period Transformations of Jacobi Functions}
\end{table}
\end{center}
These transformations can be used to transform addition theorems of one function into another. We also give the following table of periods and poles with further information in the case where periodicity is expanded to modulo $(4K,4iK')$.

\begin{center}
\begin{table}[h]
\begin{tabular}{|c|c|c|c|}
\hline
Function & Minimal Periods & Poles modulo Periods & Poles modulo $(4K,4iK')$ \\ \hline
$\sn$ & $4K,2iK'$ & $iK',iK'+2K$ & $iK',iK'+2K,3iK',3iK'+2K$ \\ \hline
$\cn$ & $4K,2K+2iK'$ &  $iK',iK'+2K$ & $iK',iK'+2K,3iK',3iK'+2K$ \\ \hline
$\dn$ & $2K,4iK'$ &  $iK',3iK'$ & $iK',iK'+2K,3iK',3iK'+2K$ \\ \hline
$\ns$ & $4K,2iK'$ & $0,2K$ & $0,2K,2iK',2K+2iK'$ \\ \hline
$\cs$ & $2K,4iK'$ & $0,2iK'$ & $0,2K,2iK',2K+2iK'$ \\ \hline
$\ds$ & $4K,2K+2iK'$ & $0,2K$ & $0,2K,2iK',2K+2iK'$ \\ \hline
$\nc$ & $4K,2K+2iK'$ & $K,3K$ & $K,3K,K+2iK',3k+2iK'$ \\ \hline
$\scJ$ & $2K,4iK'$ & $K,K+2iK'$ & $K,3K,K+2iK',3k+2iK'$ \\ \hline
$\dc$ & $4K,2iK'$ & $K,3K$ & $K,3K,K+2iK',3k+2iK'$ \\ \hline
$\nd$ & $2K,4iK'$ & $K+iK',K+3iK'$ & $K+iK',K+3iK',3K+iK',3K+3iK'$ \\ \hline
$\sd$ & $4K,2K+2iK'$ & $K+iK',3K+iK'$ & $K+iK',K+3iK',3K+iK',3K+3iK'$ \\ \hline
$\cd$ & $4K,2iK'$ & $K+iK',3K+iK'$ & $K+iK',K+3iK',3K+iK',3K+3iK'$ \\ \hline
\end{tabular}\medskip
\caption{Jacobi Functions, Their Periods and Poles}
\label{periodtable}
\end{table}
\end{center}
Notice that the sum of the poles modulo minimal periods is always a half-period. More importantly, it is always a period when expanded to modulo $(4K,4iK')$. Let $(\pq)^{(-1)}=1$ be the constant function for convenience. This allows us to generalize so that the functions $\pq^{\alpha}$ and $\pq^{(\beta)}$ are elliptic functions of orders $2\alpha$ and $2\beta+2$ with respect to their minimal periods. However, they are of order $4\alpha$ and $4\beta+4$ with respect to the periods $(4K,4iK')$. We can also divide the functions into period and pole classes. Naturally, a period class is the set of functions having equal minimal periods and can be found as $P_{\p}=P_{\qr}=\{\pn,\np,\qr,\rqJ\}$ for a permutation $\p\q\text{r}$ of $\s\cJ\dJ$. A pole class $Z_{\p}$ is determined as the set of functions that share the common poles when the periodicity is expanded to modulo $(4K,4iK')$. It can be seen that $Z_{\p}=\{\qp, \q\neq \p\}$ for any letter $\p$, as shown in the Table \ref{periodtable}. \medskip

Take $r$ of the $12$ Jacobi functions $(\pq)_1,\ldots, (\pq)_r$. Let us have distinct positive integers $\alpha_{\ell,j},\alpha_{\ell,j'}'$ where $\ell=1,\ldots, r$ and $j=1,\ldots, n_\ell, j'=1,\ldots,n'_\ell$. Furthermore, assume $\beta_{\ell,j},\beta_{\ell,j'}'\in \mathbb{Z}^+\cup\{-1\}$ where $\ell=1,\ldots, r$ and $j=1,\ldots, m_\ell, j'=1,\ldots,m'_\ell$. We can order the numbers such that $\alpha_{\ell,1}< \alpha_{\ell,2}< \ldots < \alpha_{\ell,n_\ell}$ and the respective inequalities hold for $\alpha_{\ell,j'}',\beta_{\ell,j},\beta_{\ell,j'}'$. Assume that these variables satisfy the condition
\begin{align}\label{VarCond}
    \sum_{\p\in\{\s,\cJ,\dJ,\n \}} \max_{\ell\in Z_{\p}}\left(4\alpha_ {\ell,n_\ell},4\alpha'_{\ell,n_\ell'},4\beta_{\ell,m_\ell}+4,4\beta_{\ell,m_\ell'}'+4  \right)=1+\sum_{\ell=1}^r n_\ell+m_\ell.
\end{align}
Let us be given \textit{generic} complex numbers $\mu_{\ell,j_1},\nu_{\ell,j_2}$ with $\ell=1,\ldots,r$ and $j_1=1,\ldots, n_\ell', j_2=1,\ldots,m_\ell'$. Finally, take $N=\sum_{\ell=1}^r n_\ell+m_\ell$ \textit{generic} complex numbers $z_1,\ldots, z_N$. The word generic here will be used in an implicit manner to mean parameters not satisfying any special equations. For example, in our investigations we shall make use of the assumption that $\pq(z_h)$ are distinct. This assumption imposes a very relaxed condition on $z_h$ as the function $\pq$ takes each complex value twice in a fundamental parallelogram. As another example, we shall calculate orders of elliptic functions formed by linear combinations of powers and derivatives of $\pq$. We will then take the maximum order for a pole as granted. This will be true for \textit{generic} values of parameters. Indeed, if the leading terms around a pole, say coming from $(\pq)_\ell^{\alpha}$ and $(\pq)_\ell^{(\beta)}$ or more, cancel each other to yield a non-maximal order, this would impose a linear relation between the parameters. This relation does not hold in general and may be avoided by an appropriate choice of the parameters. However, our results will hold for \textit{all} values of the parameters by the means of analytic continuation or taking limits. In fact, non-generic parameters are associated with degenerate cases of our formula, such as when a pole is encountered or a double root is present. Our proofs can be modified to include these degenerate cases. However, this is unnecessary as the process of analytic continuation provides a much simpler solution to this issue. The reader may consult \cite{Gürel} for examples of these phenomena. \medskip

We have
\begin{align*}
    \begin{vmatrix}
(\pq)_1^{\alpha_{1,1}}(z_1) & \cdots & (\pq)_1^{\alpha_{1,n_1}}(z_1) & (\pq)_1^{(\beta_{1,1})}(z_1) & \cdots & (\pq)_1^{(\beta_{1,m_1})}(z_1) & \cdots &  (\pq)_r^{(\beta_{r,m_r})}(z_1) \\
 \vdots& \cdots & \cdots &  \cdots & \cdots & \cdots & \cdots &  \vdots \\
 (\pq)_1^{\alpha_{1,1}}(z_N)&\cdots  & \cdots & \cdots & \cdots & \cdots & \cdots &  (\pq)_r^{(\beta_{r,m_r})}(z_N)
\end{vmatrix}\neq 0
\end{align*}
for generic values of $z_1,\ldots,z_N$. This guarantees the existence of complex numbers $\lambda_{\ell,j_1},\kappa_{\ell,j_2}$ with $j_1=1,\ldots,n_\ell, j_2=1,\ldots, m_\ell$ such that the equation
\begin{align*}
    &\lambda_{1,1}(\pq)_1^{\alpha_{1,1}}(z_h)+\ldots+\lambda_{1,n_1}(\pq)_1^{\alpha_{1,n_1}}(z_h)+\kappa_{1,1}(\pq)_1^{(\beta_{1,1})}(z_h)+\ldots+\kappa_{1,m_1}(\pq)_1^{(\beta_{1,m_1})}(z_h)+\ldots\\
+&\lambda_{r,1}(\pq)_r^{\alpha_{r,1}}(z_h)+\ldots+\lambda_{r,n_r}(\pq)_r^{\alpha_{r,n_r}}(z_h)+\kappa_{r,1}(\pq)_r^{(\beta_{r,1})}(z_h)+\ldots+\kappa_{r,m_r}(\pq)_r^{(\beta_{r,m_r})}(z_h)\\
=&\mu_{1,1}(\pq)_1^{\alpha_{1,1}'}(z_h)+\ldots+\mu_{1,n_1'}(\pq)_1^{\alpha_{1,n_1'}'}(z_h)+\nu_{1,1}(\pq)_1^{(\beta'_{1,1})}(z_h)+\ldots+\nu_{1,m_1'}(\pq)_1^{(\beta'_{1,m_1'})}(z_h)+\ldots\\
+&\mu_{r,1}(\pq)_r^{\alpha_{r,1}'}(z_h)+\ldots+\mu_{r,n_r'}(\pq)_r^{\alpha_{r,n_r'}'}(z_h)+\nu_{r,1}(\pq)_r^{(\beta'_{r,1})}(z_h)+\ldots+\nu_{r,m_r'}(\pq)_r^{(\beta'_{r,m_r'})}(z_h)
\end{align*}
holds for all $h=1,\ldots,N$. We remark that the constants $\lambda_{\ell,j_1}$ and $\kappa_{\ell,j_2}$ are rational functions of $\pq(z_h),\pq'(z_h)$ for $h=1,\ldots,N$ and the constants $\mu_{\ell,j_1},\nu_{\ell,j_2}$. They can be calculated \textit{explicitly} by Cramer's formula. Finally, let $z$ be defined by $z+z_1+\ldots+z_N=0$. Our main theorem follows.

\begin{theorem}\label{generaldetthm}
    The following equation holds,
    \begin{align*}
        \begin{vmatrix}
\sum_{\ell=1}^r\sum_{j=1}^{n_\ell'} \mu_{\ell,j}(\pq)_\ell^{\alpha_{\ell,j}'}(z_1)+\sum_{\ell=1}^r\sum_{j=1}^{m_\ell'} \nu_{\ell,j}(\pq)_\ell^{(\beta_{\ell,j}')}(z_1)& (\pq)_1^{\alpha_{1,1}}(z_1) &  \cdots & (\pq)_r^{(\beta_{r,m_r})}(z_1)  \\
 \vdots& \vdots & \cdots & \vdots \\
\sum_{\ell=1}^r\sum_{j=1}^{n_\ell'} \mu_{\ell,j}(\pq)_\ell^{\alpha_{\ell,j}'}(z_N)+\sum_{\ell=1}^r\sum_{j=1}^{m_\ell'} \nu_{\ell,j}(\pq)_\ell^{(\beta_{\ell,j}')}(z_N)& (\pq)_1^{\alpha_{1,1}}(z_N) &  \cdots & (\pq)_r^{(\beta_{r,m_r})}(z_N) \\
\sum_{\ell=1}^r\sum_{j=1}^{n_\ell'} \mu_{\ell,j}(\pq)_\ell^{\alpha_{\ell,j}'}(z)+\sum_{\ell=1}^r\sum_{j=1}^{m_\ell'} \nu_{\ell,j}(\pq)_\ell^{(\beta_{\ell,j}')}(z)& (\pq)_1^{\alpha_{1,1}}(z) &  \cdots & (\pq)_r^{(\beta_{r,m_r})}(z)
\end{vmatrix}=0.
    \end{align*}    
\end{theorem}
\begin{proof}
    Let us consider the function 
    \begin{align*}
        \psi(u)=&\lambda_{1,1}(\pq)_1^{\alpha_{1,1}}(u)+\ldots+\lambda_{1,n_1}(\pq)_1^{\alpha_{1,n_1}}(u)+\kappa_{1,1}(\pq)_1^{(\beta_{1,1})}(u)+\ldots+\kappa_{1,m_1}(\pq)_1^{(\beta_{1,m_1})}(u)+\ldots\\
+&\lambda_{r,1}(\pq)_r^{\alpha_{r,1}}(u)+\ldots+\lambda_{r,n_r}(\pq)_r^{\alpha_{r,n_r}}(u)+\kappa_{r,1}(\pq)_r^{(\beta_{r,1})}(u)+\ldots+\kappa_{r,m_r}(\pq)_r^{(\beta_{r,m_r})}(u)\\
-&\mu_{1,1}(\pq)_1^{\alpha_{1,1}'}(u)-\ldots-\mu_{1,n_1'}(\pq)_1^{\alpha_{1,n_1'}'}(u)-\nu_{1,1}(\pq)_1^{(\beta'_{1,1})}(u)-\ldots-\nu_{1,m_1'}(\pq)_1^{(\beta'_{1,m_1'})}(u)-\ldots\\
-&\mu_{r,1}(\pq)_r^{\alpha_{r,1}'}(u)-\ldots-\mu_{r,n_r'}(\pq)_r^{\alpha_{r,n_r'}'}(u)-\nu_{r,1}(\pq)_r^{(\beta'_{r,1})}(u)-\ldots-\nu_{r,m_r'}(\pq)_r^{(\beta'_{r,m_r'})}(u).
    \end{align*}
It is an elliptic function of periods $4K,4iK'$ with the order
\begin{align*}
        \sum_{\p\in\{\s,\cJ,\dJ,\n \}} \max_{\ell\in Z_{\p}}\left(4\alpha_ {\ell,n_\ell},4\alpha'_{\ell,n_\ell'},4\beta_{\ell,m_\ell}+4,4\beta_{\ell,m_\ell'}'+4  \right)=N+1
\end{align*}
for generic parameters. Assuming that the numbers $z_h$ are distinct, we have $N$ distinct zeros already. Let $w$ be the last zero. By Table \ref{periodtable}, the poles sum to zero modulo $(4K,4iK')$. Thus Abel's theorem gives us $w+z_1+z_2+\ldots+z_N\equiv 0 \bmod (4K,4iK')$ and $w\equiv z \bmod (4K,4iK')$. Therefore we have
\begin{align*}
        &\lambda_{1,1}(\pq)_1^{\alpha_{1,1}}(z)+\ldots+\lambda_{1,n_1}(\pq)_1^{\alpha_{1,n_1}}(z)+\kappa_{1,1}(\pq)_1^{(\beta_{1,1})}(z)+\ldots+\kappa_{1,m_1}(\pq)_1^{(\beta_{1,m_1})}(z)+\ldots\\
+&\lambda_{r,1}(\pq)_r^{\alpha_{r,1}}(z)+\ldots+\lambda_{r,n_r}(\pq)_r^{\alpha_{r,n_r}}(z)+\kappa_{r,1}(\pq)_r^{(\beta_{r,1})}(z)+\ldots+\kappa_{r,m_r}(\pq)_r^{(\beta_{r,m_r})}(z)\\
=&\mu_{1,1}(\pq)_1^{\alpha_{1,1}'}(z)+\ldots+\mu_{1,n_1'}(\pq)_1^{\alpha_{1,n_1'}'}(z)+\nu_{1,1}(\pq)_1^{(\beta'_{1,1})}(z)+\ldots+\nu_{1,m_1'}(\pq)_1^{(\beta'_{1,m_1'})}(z)+\ldots\\
+&\mu_{r,1}(\pq)_r^{\alpha_{r,1}'}(z)+\ldots+\mu_{r,n_r'}(\pq)_r^{\alpha_{r,n_r'}'}(z)+\nu_{r,1}(\pq)_r^{(\beta'_{r,1})}(z)+\ldots+\nu_{r,m_r'}(\pq)_r^{(\beta'_{r,m_r'})}(z).
\end{align*}
Writing these equations in the matrix form gives us
\begin{align*}
    &\begin{pmatrix}
\sum_{\ell=1}^r\sum_{j=1}^{n_\ell'} \mu_{\ell,j}(\pq)_\ell^{\alpha_{\ell,j}'}(z_1)+\sum_{\ell=1}^r\sum_{j=1}^{m_\ell'} \nu_{\ell,j}(\pq)_\ell^{(\beta_{\ell,j}')}(z_1)& (\pq)_1^{\alpha_{1,1}}(z_1) &  \cdots & (\pq)_r^{(\beta_{r,m_r})}(z_1)  \\
 \vdots& \vdots & \cdots & \vdots \\
\sum_{\ell=1}^r\sum_{j=1}^{n_\ell'} \mu_{\ell,j}(\pq)_\ell^{\alpha_{\ell,j}'}(z_N)+\sum_{\ell=1}^r\sum_{j=1}^{m_\ell'} \nu_{\ell,j}(\pq)_\ell^{(\beta_{\ell,j}')}(z_N)& (\pq)_1^{\alpha_{1,1}}(z_N) &  \cdots & (\pq)_r^{(\beta_{r,m_r})}(z_N) \\
\sum_{\ell=1}^r\sum_{j=1}^{n_\ell'} \mu_{\ell,j}(\pq)_\ell^{\alpha_{\ell,j}'}(z)+\sum_{\ell=1}^r\sum_{j=1}^{m_\ell'} \nu_{\ell,j}(\pq)_\ell^{(\beta_{\ell,j}')}(z)& (\pq)_1^{\alpha_{1,1}}(z) &  \cdots & (\pq)_r^{(\beta_{r,m_r})}(z)
\end{pmatrix}\\
\cdot&\begin{pmatrix}
-1 \\
 \lambda_{1,1}\\
\vdots\\
\kappa_{r,m_r}
\end{pmatrix}=\begin{pmatrix}
\psi(z_1) \\
\vdots\\
\psi(z_N)\\
\psi(z)
\end{pmatrix}=0.
\end{align*}
Since the vector on the left-hand side is non-zero, the matrix on the left-hand side is singular and has a vanishing determinant. This concludes the proof.
\end{proof}

In the subsequent subsections, we will focus on special cases of Theorem \ref{generaldetthm}. Specifically, we will consider uniform powers, where the numbers $\alpha_{\ell,j},\beta_{\ell,j}$ are equal for all $\ell$ in an ascending order and the powers $\alpha_{\ell,j'}',\beta'_{\ell,j'}$ are trivial.

\subsection{The Case of a Single Pole Class and Uniform Powers}
Here we shall consider the following special case
\begin{align*}
    1&=\lambda_{1,1}(\pq)_1(z_h)+\ldots+\lambda_{1,n}(\pq)_1^{n}(z_h)+\kappa_{1,1}(\pq)_1'(z_h)+\ldots+\kappa_{1,m}(\pq)_1^{(m)}(z_h)+\ldots\\
&+\lambda_{r,1}(\pq)_r(z_h)+\ldots+\lambda_{r,n}(\pq)_r^{n}(z_h)+\kappa_{r,1}(\pq)_r'(z_h)+\ldots+\kappa_{r,m}(\pq)_r^{(m)}(z_h)
\end{align*}
where the functions $(\pq)_r$ belong to only a single pole class. Our assumption on the variables yields
\begin{align*}
    4\max(n,m+1)=r(n+m)+1.
\end{align*}
It follows immediately that $r=3$ and we have the solutions 
\begin{align*}
    (n,m,N)=(t,3t+1,12t+3) \text{ or }(t+1,3t,12t+3)
\end{align*}
for a natural number $t$. Notice that a single pole class is given by $\p_1 \q, \p_2\q,\p_3\q$ where $\p_1,\p_2,\p_3,\q$ is a permutation of $\s \cJ \dJ \n$. We may simply write $(\pq)_\ell=\p_\ell\q$.
\begin{theorem}
    If the functions $(\pq)_1,(\pq)_2, (\pq)_3$ are the elements of a single pole class, then the following equation holds,
    \begin{align*}
        0=&\begin{vmatrix}
1&(\pq)_1(z_1) & \cdots & (\pq)_1^{t}(z_1) & (\pq)_1'(z_1) & \cdots & (\pq)_1^{(3t+1)}(z_1) & \cdots &   (\pq)_3^{(3t+1)}(z_1) \\
 \vdots&\vdots & \cdots & \cdots &  \cdots & \cdots & \cdots & \cdots & \vdots \\
 1& (\pq)_1(z_{12t+3})&\cdots  & \cdots & \cdots & \cdots & \cdots & \cdots   & (\pq)_3^{(3t+1)}(z_{12t+3})\\
1& (\pq)_1(z)&\cdots &\cdots &\cdots &\cdots &\cdots &\cdots & (\pq)_3^{(3t+1)}(z)
\end{vmatrix}\\
=&\begin{vmatrix}
1&(\pq)_1(z_1) & \cdots & (\pq)_1^{t+1}(z_1) & (\pq)_1'(z_1) & \cdots & (\pq)_1^{(3t)}(z_1) & \cdots &   (\pq)_3^{(3t)}(z_1) \\
 \vdots&\vdots & \cdots & \cdots &  \cdots & \cdots & \cdots & \cdots & \vdots \\
 1& (\pq)_1(z_{12t+3})&\cdots  & \cdots & \cdots & \cdots & \cdots & \cdots   & (\pq)_3^{(3t)}(z_{12t+3})\\
1& (\pq)_1(z)&\cdots &\cdots &\cdots &\cdots &\cdots &\cdots & (\pq)_3^{(3t)}(z)
\end{vmatrix}.
    \end{align*}
\end{theorem}
\begin{example}
    Taking $t=0$ and the pole class $Z_{\n}=\{\sn,\cn,\dn\}$, we obtain the formulas
    \begin{align*}
        0=\begin{vmatrix}
1 & \sn u & \cn u & \dn u \\
1 & \sn v & \cn v & \dn v \\
1 & \sn w & \cn w & \dn w \\
1 & \sn z & \cn z & \dn z 
\end{vmatrix}=\begin{vmatrix}
1 & \sn u \cn u & \sn u \dn u & \cn u \dn u \\
1 & \sn v \cn v & \sn v \dn v & \cn v \dn v \\
1 & \sn w \cn w & \sn w \dn w & \cn w \dn w \\
1 & \sn z \cn z & \sn z \dn z & \cn z \dn z
\end{vmatrix}
    \end{align*}
where $u+v+w+z=0$, the first formula being Cayley's celebrated addition theorem. Similarly, the remaining pole classes $Z_{\s},Z_{\cJ}$ and $Z_{\dJ}$ yield
\begin{align*}
    0&=\begin{vmatrix}
1 & \cs u \ns u & \ds u \ns u & \cs u \ds u \\
1 & \cs v \ns v & \ds v \ns v & \cs v \ds v \\
1 & \cs w \ns w & \ds w \ns w & \cs w \ds w \\
1 & \cs z \ns z & \ds z \ns z & \cs z \ds z \\
\end{vmatrix}\\
&=\begin{vmatrix}
1 & \scJ u \nc u & \dc u \nc u & \scJ u \dc u \\
1 & \scJ v \nc v & \dc v \nc v & \scJ v \dc v \\
1 & \scJ w \nc w & \dc w \nc w & \scJ w \dc w \\
1 & \scJ z \nc z & \dc z \nc z & \scJ z \dc z 
\end{vmatrix}\\
&=\begin{vmatrix}
1 & \sd u \nd u & \cd u \nd u & \sd u \cd u \\
1 & \sd v \nd v & \cd v \nd v & \sd v \cd v \\
1 & \sd w \nd w & \cd w \nd w & \sd w \cd w \\
1 & \sd z \nd z & \cd z \nd z & \sd z \cd z 
\end{vmatrix}.
\end{align*}
We have not written the cases $n=1, m=0$ again since they are equivalent to Cayley's addition theorem. Notice that $(\pq)_1'$ is a constant multiple of $(\pq)_2(\pq)_3$. We can collect these formulas together by writing
\begin{align*}
    \begin{vmatrix}
(\pn)_1^2 (u) & (\pn)_2 (u) & (\pn)_3 (u) & (\pn)_2(u)(\pn)_3(u) \\
(\pn)_1^2 (v) & (\pn)_2 (v) & (\pn)_3 (v) & (\pn)_2(v)(\pn)_3(v) \\
(\pn)_1^2 (w) & (\pn)_2 (w) & (\pn)_3 (w) & (\pn)_2(w)(\pn)_3(w) \\
(\pn)_1^2 (z) & (\pn)_2 (z) & (\pn)_3 (z) & (\pn)_2(z)(\pn)_3(z) 
\end{vmatrix}=0.
\end{align*}
\end{example}

\begin{example}
    Taking $t=1$ and the pole class $Z_{\n}$, we obtain the formulas
    \begin{align*}
        0=&\begin{vmatrix}
1&\sn(z_1) & \sn'(z_1) & \cdots & \sn^{(4)}(z_1) &\cn(z_1) &\cdots &   \dn^{(4)}(z_1) \\
 \vdots&\vdots & \cdots  &  \cdots & \cdots & \cdots & \cdots & \vdots \\
1&\sn(z_{15}) & \sn'(z_{15}) & \cdots & \sn^{(4)}(z_{15}) &\cn(z_{15}) &\cdots &   \dn^{(4)}(z_{15}) \\
1&\sn(z) & \sn'(z) & \cdots & \sn^{(4)}(z) &\cn(z) &\cdots &   \dn^{(4)}(z) 
\end{vmatrix}\\
=&\begin{vmatrix}
1&\sn(z_1) &\sn^2(z_1) & \sn'(z_1) & \sn''(z_1) & \sn^{(3)}(z_1) &\cn(z_1) &\cdots &   \dn^{(3)}(z_1) \\
 \vdots&\vdots & \cdots  &  \cdots & \cdots & \cdots & \cdots &\cdots & \vdots \\
1&\sn(z_{15}) &\sn^2(z_{15}) & \sn'(z_{15}) & \sn''(z_{15}) & \sn^{(3)}(z_{15}) &\cn(z_{15}) &\cdots &   \dn^{(3)}(z_{15}) \\
1&\sn(z) &\sn^2(z) & \sn'(z) & \sn''(z) & \sn^{(3)}(z) &\cn(z) &\cdots &   \dn^{(3)}(z) \\
\end{vmatrix}.
    \end{align*}
Similarly, the pole classes $Z_{\s},Z_{\cJ}$ and $Z_{\dJ}$ give us
\begin{align*}
     0=&\begin{vmatrix}
1&\cs(z_1) & \cs'(z_1) & \cdots & \cs^{(4)}(z_1) &\ds(z_1) &\cdots &   \ns^{(4)}(z_1) \\
 \vdots&\vdots & \cdots  &  \cdots & \cdots & \cdots & \cdots & \vdots \\
1&\cs(z_{15}) & \cs'(z_{15}) & \cdots & \cs^{(4)}(z_{15}) &\ds(z_{15}) &\cdots &   \ns^{(4)}(z_{15}) \\
1&\cs(z) & \cs'(z) & \cdots & \cs^{(4)}(z) &\ds(z) &\cdots &   \ns^{(4)}(z)
\end{vmatrix}\\
=&\begin{vmatrix}
1&\cs(z_1) &\cs^2(z_1) & \cs'(z_1) & \cs''(z_1) & \cs^{(3)}(z_1) &\ds(z_1) &\cdots &   \ns^{(3)}(z_1) \\
 \vdots&\vdots & \cdots  &  \cdots & \cdots & \cdots & \cdots &\cdots & \vdots \\
1&\cs(z_{15}) &\cs^2(z_{15}) & \cs'(z_{15}) & \cs''(z_{15}) & \cs^{(3)}(z_{15}) &\ds(z_{15}) &\cdots &   \ns^{(3)}(z_{15}) \\
1&\cs(z) &\cs^2(z) & \cs'(z) & \cs''(z) & \cs^{(3)}(z) &\ds(z) &\cdots &   \ns^{(3)}(z) \\
\end{vmatrix}
\end{align*}
and
\begin{align*}
    0=&\begin{vmatrix}
1&\scJ(z_1) & \scJ'(z_1) & \cdots & \scJ^{(4)}(z_1) &\dc(z_1) &\cdots &   \nc^{(4)}(z_1) \\
 \vdots&\vdots & \cdots  &  \cdots & \cdots & \cdots & \cdots & \vdots \\
1&\scJ(z_{15}) & \scJ'(z_{15}) & \cdots & \scJ^{(4)}(z_{15}) &\dc(z_{15}) &\cdots &   \nc^{(4)}(z_{15}) \\
1&\scJ(z) & \scJ'(z) & \cdots & \scJ^{(4)}(z) &\dc(z) &\cdots &   \nc^{(4)}(z)
\end{vmatrix}\\
=&\begin{vmatrix}
1&\scJ(z_1) &\scJ^2(z_1) & \scJ'(z_1) &\scJ''(z_1) & \scJ^{(3)}(z_1) &\dc(z_1) &\cdots &   \nc^{(3)}(z_1) \\
 \vdots&\vdots & \cdots  &  \cdots & \cdots & \cdots & \cdots &\cdots & \vdots \\
1&\scJ(z_{15}) &\scJ^2(z_{15}) & \scJ'(z_{15}) &\scJ''(z_{15}) & \scJ^{(3)}(z_{15}) &\dc(z_{15}) &\cdots &   \nc^{(3)}(z_{15}) \\
1&\scJ(z) &\scJ^2(z) & \scJ'(z) &\scJ''(z) & \scJ^{(3)}(z) &\dc(z) &\cdots &   \nc^{(3)}(z) \\
\end{vmatrix}
\end{align*}
and
\begin{align*}
    0=&\begin{vmatrix}
1&\sd(z_1) & \sd'(z_1) & \cdots & \sd^{(4)}(z_1) &\cd(z_1) &\cdots &   \nd^{(4)}(z_1) \\
 \vdots&\vdots & \cdots  &  \cdots & \cdots & \cdots & \cdots & \vdots \\
1&\sd(z_{15}) & \sd'(z_{15}) & \cdots & \sd^{(4)}(z_{15}) &\cd(z_{15}) &\cdots &   \nd^{(4)}(z_{15}) \\
1&\sd(z) & \sd'(z) & \cdots & \sd^{(4)}(z) &\cd(z) &\cdots &   \nd^{(4)}(z) 
\end{vmatrix}\\
=&\begin{vmatrix}
1&\sd(z_1) &\sd^2(z_1) & \sd'(z_1) &\sd''(z_1) & \sd^{(3)}(z_1) &\cd(z_1) &\cdots &   \nd^{(3)}(z_1) \\
 \vdots&\vdots & \cdots  &  \cdots & \cdots & \cdots & \cdots &\cdots & \vdots \\
1&\sd(z_{15}) &\sd^2(z_{15}) & \sd'(z_{15}) &\sd''(z_{15}) & \sd^{(3)}(z_{15}) &\cd(z_{15}) &\cdots &   \nd^{(3)}(z_{15}) \\
1&\sd(z) &\sd^2(z) & \sd'(z) &\sd''(z) & \sd^{(3)}(z) &\cd(z) &\cdots &   \nd^{(3)}(z) \\
\end{vmatrix}.
\end{align*}
\end{example}
\begin{example}
    Taking $t=2$ and the pole class $Z_{\n}$, we obtain the formulas
    \begin{align*}
        0=&\begin{vmatrix}
1&\sn(z_1) &  \sn^{2}(z_1) & \sn'(z_1) & \cdots & \sn^{(7)}(z_1) & \cn(z_1)&\cdots &   \dn^{(7)}(z_1) \\
 \vdots&\vdots & \cdots & \cdots &  \cdots & \cdots & \cdots &\cdots & \vdots \\
1&\sn(z_{27}) &  \sn^{2}(z_{27}) & \sn'(z_{27}) & \cdots & \sn^{(7)}(z_{27}) & \cn(z_{27})&\cdots &   \dn^{(7)}(z_{27}) \\
1&\sn(z) &  \sn^{2}(z) & \sn'(z) & \cdots & \sn^{(7)}(z)&\cn(z) & \cdots &   \dn^{(7)}(z) 
\end{vmatrix}\\
=&\begin{vmatrix}
1&\sn(z_1) & \sn^2(z_1) & \sn^{3}(z_1) & \sn'(z_1) & \cdots & \sn^{(6)}(z_1) &\cn(z_1) & \cdots &   \dn^{(6)}(z_1) \\
 \vdots&\vdots & \cdots & \cdots &  \cdots & \cdots & \cdots & \cdots & \cdots & \vdots \\
1&\sn(z_{27}) & \sn^2(z_{27}) & \sn^{3}(z_{27}) & \sn'(z_{27}) & \cdots & \sn^{(6)}(z_{27}) &\cn(z_{27}) & \cdots &   \dn^{(6)}(z_{27}) \\
1&\sn(z) & \sn^2(z) & \sn^{3}(z) & \sn'(z) & \cdots & \sn^{(6)}(z) &\cn(z) & \cdots &   \dn^{(6)}(z) \\
\end{vmatrix}.
    \end{align*}
\end{example}

\subsection{The Case of Two Pole Classes and Uniform Powers}
Here we shall consider the same special case
\begin{align*}
    1&=\lambda_{1,1}(\pq)_1(z_h)+\ldots+\lambda_{1,n}(\pq)_1^{n}(z_h)+\kappa_{1,1}(\pq)_1'(z_h)+\ldots+\kappa_{1,m}(\pq)_1^{(m)}(z_h)+\ldots\\
&+\lambda_{r,1}(\pq)_r(z_h)+\ldots+\lambda_{r,n}(\pq)_r^{n}(z_h)+\kappa_{r,1}(\pq)_r'(z_h)+\ldots+\kappa_{r,m}(\pq)_r^{(m)}(z_h)
\end{align*}
where the functions $(\pq)_r$ span two pole classes. Our assumption on the variables yields
\begin{align*}
    8\max(n,m+1)=r(n+m)+1.
\end{align*}
It is obvious that $r$ is odd and $r> 4$ because $2\max(x,y)\ge x+y$. Since $r\le 6$ as there are 6 functions of two pole classes, this forces $r=5$. The solutions of the equation $8\max(n,m+1)=5(n+m)+1$ are given by
\begin{align*}
    (n,m,N)=(5 t+2,3t+1,40t+15) \text{ or }(3t+2,5t+1,40t+15).
\end{align*}
\begin{theorem}
    If the functions $(\pq)_1,\ldots, (\pq)_5$ span two pole classes, then the following equation holds,
    \begin{align*}
        0=&\begin{vmatrix}
1&(\pq)_1(z_1) & \cdots & (\pq)_1^{5t+2}(z_1) & (\pq)_1'(z_1) & \cdots & (\pq)_1^{(3t+1)}(z_1) & \cdots &   (\pq)_5^{(3t+1)}(z_1) \\
 \vdots&\vdots & \cdots & \cdots &  \cdots & \cdots & \cdots & \cdots & \vdots \\
 1& (\pq)_1(z_{40t+15})&\cdots  & \cdots & \cdots & \cdots & \cdots & \cdots   & (\pq)_5^{(3t+1)}(z_{40t+15})\\
1& (\pq)_1(z)&\cdots &\cdots &\cdots &\cdots &\cdots &\cdots & (\pq)_5^{(3t+1)}(z)
\end{vmatrix}\\
=&\begin{vmatrix}
1&(\pq)_1(z_1) & \cdots & (\pq)_1^{3t+2}(z_1) & (\pq)_1'(z_1) & \cdots & (\pq)_1^{(5t+1)}(z_1) & \cdots &   (\pq)_5^{(5t+1)}(z_1) \\
 \vdots&\vdots & \cdots & \cdots &  \cdots & \cdots & \cdots & \cdots & \vdots \\
 1& (\pq)_1(z_{40t+15})&\cdots  & \cdots & \cdots & \cdots & \cdots & \cdots   & (\pq)_5^{(5t+1)}(z_{40t+15})\\
1& (\pq)_1(z)&\cdots &\cdots &\cdots &\cdots &\cdots &\cdots & (\pq)_5^{(5t+1)}(z)
\end{vmatrix}.
    \end{align*}
\end{theorem}
An easy calculation shows that there are $36$ possible combinations of these 5 functions.
\begin{example}
    Taking $t=0$, we obtain the formula
    \begin{align*}
        0=\begin{vmatrix}
1 & \sn(z_1) & \sn^2(z_1) & \sn'(z_1) & \cn(z_1) & \cdots & \cd '(z_1) \\ 
\vdots & \vdots  & \cdots & \cdots & \cdots & \cdots & \vdots \\
1 & \sn(z_{15})  & \cdots & \cdots & \cdots & \cdots & \cd '(z_{15}) \\
1 & \sn(z) &  \cdots & \cdots & \cdots & \cdots & \cd '(z)
\end{vmatrix}
    \end{align*}
    where the determinant includes the functions $\sn,\cn,\dn,\sd,\cd$.
\end{example}
\begin{example}
    Taking $t=1$, we obtain the formulas
    \begin{align*}
        0=&\begin{vmatrix}
1&\ns(z_1) & \cdots & \ns^{7}(z_1) & \ns'(z_1) & \cdots & \ns^{(4)}(z_1) & \cdots &  \scJ^{(4)}(z_1) \\
 \vdots&\vdots & \cdots & \cdots &  \cdots & \cdots & \cdots & \cdots & \vdots \\
 1& \ns(z_{55})&\cdots  & \cdots & \cdots & \cdots & \cdots & \cdots   & \scJ^{(4)}(z_{55})\\
1& \ns(z)&\cdots &\cdots &\cdots &\cdots &\cdots &\cdots & \scJ^{(4)}(z)
\end{vmatrix}\\
=&\begin{vmatrix}
1&\ns(z_1) & \cdots & \ns^{5}(z_1) & \ns'(z_1) & \cdots & \ns^{(4)}(z_1) & \cdots &  \scJ^{(6)}(z_1) \\
 \vdots&\vdots & \cdots & \cdots &  \cdots & \cdots & \cdots & \cdots & \vdots \\
 1& \ns(z_{55})&\cdots  & \cdots & \cdots & \cdots & \cdots & \cdots   & \scJ^{(6)}(z_{55})\\
1& \ns(z)&\cdots &\cdots &\cdots &\cdots &\cdots &\cdots & \scJ^{(6)}(z)
\end{vmatrix}
    \end{align*}
    where the determinants include the functions $\ns,\ds,\nc,\dc,\scJ$.
\end{example}

\subsection{The Case of Three Pole Classes and Uniform Powers}
Here we shall consider the same special case
\begin{align*}
    1&=\lambda_{1,1}(\pq)_1(z_h)+\ldots+\lambda_{1,n}(\pq)_1^{n}(z_h)+\kappa_{1,1}(\pq)_1'(z_h)+\ldots+\kappa_{1,m}(\pq)_1^{(m)}(z_h)+\ldots\\
&+\lambda_{r,1}(\pq)_r(z_h)+\ldots+\lambda_{r,n}(\pq)_r^{n}(z_h)+\kappa_{r,1}(\pq)_r'(z_h)+\ldots+\kappa_{r,m}(\pq)_r^{(m)}(z_h)
\end{align*}
where the functions $(\pq)_r$ span three pole classes. Our assumption on the variables yields
\begin{align*}
    12\max(n,m+1)=r(n+m)+1.
\end{align*}
An analogous discussion gives us $r=7$ with the solutions
\begin{align*}
        (n,m,N)=(7t+3,5t+2,84t+35) \text{ or }(5t+3,7t+2,84t+35).
\end{align*}
\begin{theorem}
    If the functions $(\pq)_1,\ldots, (\pq)_7$ span three pole classes, then the following equation holds,
    \begin{align*}
        0=&\begin{vmatrix}
1&(\pq)_1(z_1) & \cdots & (\pq)_1^{7t+3}(z_1) & (\pq)_1'(z_1) & \cdots & (\pq)_1^{(5t+2)}(z_1) & \cdots &   (\pq)_7^{(5t+2)}(z_1) \\
 \vdots&\vdots & \cdots & \cdots &  \cdots & \cdots & \cdots & \cdots & \vdots \\
 1& (\pq)_1(z_{84t+35})&\cdots  & \cdots & \cdots & \cdots & \cdots & \cdots   & (\pq)_7^{(5t+2)}(z_{84t+35})\\
1& (\pq)_1(z)&\cdots &\cdots &\cdots &\cdots &\cdots &\cdots & (\pq)_7^{(5t+2)}(z)
\end{vmatrix}\\
=&\begin{vmatrix}
1&(\pq)_1(z_1) & \cdots & (\pq)_1^{5t+3}(z_1) & (\pq)_1'(z_1) & \cdots & (\pq)_1^{(7t+2)}(z_1) & \cdots &   (\pq)_7^{(7t+2)}(z_1) \\
 \vdots&\vdots & \cdots & \cdots &  \cdots & \cdots & \cdots & \cdots & \vdots \\
 1& (\pq)_1(z_{84t+35})&\cdots  & \cdots & \cdots & \cdots & \cdots & \cdots   & (\pq)_7^{(7t+2)}(z_{84t+35})\\
1& (\pq)_1(z)&\cdots &\cdots &\cdots &\cdots &\cdots &\cdots & (\pq)_7^{(7t+2)}(z)
\end{vmatrix}.
    \end{align*}
\end{theorem}
There are $144$ combinations for the choice of functions.
\begin{example}
    Taking $t=0$, we obtain the formula
    \begin{align*}
        0=&\begin{vmatrix}
1&\sn(z_1) & \sn^2(z_1) & \sn^{3}(z_1) & \sn'(z_1) & \sn''(z_1) & \cdots &   \nd''(z_1) \\
 \vdots&\vdots & \cdots & \cdots &  \cdots &  \cdots & \cdots & \vdots \\
 1& \sn(z_{35})&\cdots  &  \cdots & \cdots & \cdots & \cdots   & \nd''(z_{35})\\
1& \sn(z)&\cdots &\cdots &\cdots  &\cdots &\cdots & \nd''(z)
\end{vmatrix}
    \end{align*}
    where the determinant includes the functions $\sn,\cn,\ns,\nc,\sd,\cd,\nd$.
\end{example}

\begin{example}
    Taking $t=1$, we obtain the formulas
    \begin{align*}
        0=&\begin{vmatrix}
1&\sn(z_1) & \cdots & \sn^{10}(z_1) &\sn'(z_1) & \cdots & \sn^{(7)}(z_1) & \cdots &   \dn^{(7)}(z_1) \\
 \vdots&\vdots & \cdots & \cdots &  \cdots & \cdots & \cdots & \cdots & \vdots \\
 1& \sn(z_{119})&\cdots  & \cdots & \cdots & \cdots & \cdots & \cdots   & \dn^{(7)}(z_{119})\\
1& \sn(z)&\cdots &\cdots &\cdots &\cdots &\cdots &\cdots & \dn^{(7)}(z)
\end{vmatrix}\\
=&\begin{vmatrix}
1&\sn(z_1) & \cdots & \sn^{8}(z_1) &\sn'(z_1) & \cdots & \sn^{(9)}(z_1) & \cdots &   \dn^{(9)}(z_1) \\
 \vdots&\vdots & \cdots & \cdots &  \cdots & \cdots & \cdots & \cdots & \vdots \\
 1& \sn(z_{119})&\cdots  & \cdots & \cdots & \cdots & \cdots & \cdots   & \dn^{(9)}(z_{119})\\
1& \sn(z)&\cdots &\cdots &\cdots &\cdots &\cdots &\cdots & \dn^{(9)}(z)
\end{vmatrix}
    \end{align*}
    where the determinants include the functions $\sn,\cn,\ns,\nc,\sd,\cd,\dn$.    
\end{example}
\subsection{The Case of Four Pole Classes and Uniform Powers}
Here we shall consider the same special case
\begin{align*}
    1&=\lambda_{1,1}(\pq)_1(z_h)+\ldots+\lambda_{1,n}(\pq)_1^{n}(z_h)+\kappa_{1,1}(\pq)_1'(z_h)+\ldots+\kappa_{1,m}(\pq)_1^{(m)}(z_h)+\ldots\\
&+\lambda_{r,1}(\pq)_r(z_h)+\ldots+\lambda_{r,n}(\pq)_r^{n}(z_h)+\kappa_{r,1}(\pq)_r'(z_h)+\ldots+\kappa_{r,m}(\pq)_r^{(m)}(z_h)
\end{align*}
where the functions $(\pq)_r$ span all four of the pole classes. Our assumption on the variables yields
\begin{align*}
    16\max(n,m+1)=r(n+m)+1.
\end{align*}
Interestingly enough, we get two different possibilities $r=9,11$ in this case. In the case $r=9$, we have the solutions
\begin{align*}
    r=9;\qquad (n,m,N)=(9t+4,7t+3,144t+63) \text{ or }(9t+3,7t+4,144t+63).
\end{align*}
\begin{theorem}
    If the functions $(\pq)_1,\ldots, (\pq)_9$ span all four of the pole classes, then the following equation holds,
    \begin{align*}
        0=&\begin{vmatrix}
1&(\pq)_1(z_1) & \cdots & (\pq)_1^{9t+4}(z_1) & (\pq)_1'(z_1) & \cdots & (\pq)_1^{(7t+3)}(z_1) & \cdots &   (\pq)_9^{(7t+3)}(z_1) \\
 \vdots&\vdots & \cdots & \cdots &  \cdots & \cdots & \cdots & \cdots & \vdots \\
 1& (\pq)_1(z_{144t+63})&\cdots  & \cdots & \cdots & \cdots & \cdots & \cdots   & (\pq)_9^{(7t+3)}(z_{144t+63})\\
1& (\pq)_1(z)&\cdots &\cdots &\cdots &\cdots &\cdots &\cdots & (\pq)_9^{(7t+3)}(z)
\end{vmatrix}\\
=&\begin{vmatrix}
1&(\pq)_1(z_1) & \cdots & (\pq)_1^{7t+4}(z_1) & (\pq)_1'(z_1) & \cdots & (\pq)_1^{(9t+3)}(z_1) & \cdots &   (\pq)_9^{(9t+3)}(z_1) \\
 \vdots&\vdots & \cdots & \cdots &  \cdots & \cdots & \cdots & \cdots & \vdots \\
 1& (\pq)_1(z_{144t+63})&\cdots  & \cdots & \cdots & \cdots & \cdots & \cdots   & (\pq)_9^{(9t+3)}(z_{144t+63})\\
1& (\pq)_1(z)&\cdots &\cdots &\cdots &\cdots &\cdots &\cdots & (\pq)_9^{(9t+3)}(z)
\end{vmatrix}.
    \end{align*}
\end{theorem}
There are $216$ combinations for the choice of functions.\medskip

In the case $r=11$, we obtain
\begin{align*}
    r=11;\qquad (n,m,N)=(11t+9,5t+4,176t+143) \text{ or }(5t+5,11t+8,176t+143).
\end{align*}
\begin{theorem}
    The following equation holds,
    \begin{align*}
        0=&\begin{vmatrix}
1&(\pq)_1(z_1) & \cdots & (\pq)_1^{11t+9}(z_1) & (\pq)_1'(z_1) & \cdots & (\pq)_1^{(5t+4)}(z_1) & \cdots &   (\pq)_{11}^{(5t+4)}(z_1) \\
 \vdots&\vdots & \cdots & \cdots &  \cdots & \cdots & \cdots & \cdots & \vdots \\
 1& (\pq)_1(z_{176t+143})&\cdots  & \cdots & \cdots & \cdots & \cdots & \cdots   & (\pq)_{11}^{(5t+4)}(z_{176t+143})\\
1& (\pq)_1(z)&\cdots &\cdots &\cdots &\cdots &\cdots &\cdots & (\pq)_{11}^{(5t+4)}(z)
\end{vmatrix}\\
=&\begin{vmatrix}
1&(\pq)_1(z_1) & \cdots & (\pq)_1^{5t+5}(z_1) & (\pq)_1'(z_1) & \cdots & (\pq)_1^{(11t+8)}(z_1) & \cdots &   (\pq)_{11}^{(11t+8)}(z_1) \\
 \vdots&\vdots & \cdots & \cdots &  \cdots & \cdots & \cdots & \cdots & \vdots \\
 1& (\pq)_1(z_{176t+143})&\cdots  & \cdots & \cdots & \cdots & \cdots & \cdots   & (\pq)_{11}^{(11t+8)}(z_{176t+143})\\
1& (\pq)_1(z)&\cdots &\cdots &\cdots &\cdots &\cdots &\cdots & (\pq)_{11}^{(11t+8)}(z)
\end{vmatrix}.
    \end{align*}
\end{theorem}
It is obvious that there are $12$ combinations for the choice of functions.

\section{Formulas for a Single Function and Explicit Addition Theorems}

\subsection{Determinant Theorems for a Single Function} 
We now give an analogue of the main theorem where all functions belong to the same period class. In this case, we do not need to expand the periodicity of the functions $(\pq)_\ell$ to $4K, 4iK'$. Each one of the $r\le 4$ functions belong to a different pole class. It can be easily verified that under the same notation, we must remove the assumption $\eqref{VarCond}$ and replace it with
\begin{align*}
    \sum_{\ell=1}^r \max\left(2\alpha_ {\ell,n_\ell},2\alpha'_{\ell,n_\ell'},2\beta_{\ell,m_\ell}+2,2\beta_{\ell,m_\ell'}'+2  \right)=1+\sum_{\ell=1}^r n_\ell+m_\ell.
\end{align*}
It is worthwhile to note that this equation is quite restrictive. Certainly, we have the inequalities $\alpha_{\ell,n_\ell}\ge n_\ell, \beta_{\ell,m_\ell}\ge  m_\ell -1$ and thus $\max(2\alpha_{\ell,n_\ell},2\beta_{\ell,m_\ell}+2)\ge n_\ell+m_\ell$ where the equality $\beta_{\ell,m_\ell}=m_\ell -1$ may hold for \textit{at most one} value of $\ell$. This is enough to force $r=1,2$. Even further considerations with the above inequalities show that we obtain no non-trivial determinants for $r=2$. We are thus obligated to have only a single function. Although it is tedious to do so, one may find all solutions of the above equation for $r=1$. We have deemed it unnecessary to record such a list here. We will investigate the system
\begin{align*}
        1&=\lambda_{1}\pq(z_h)+\ldots+\lambda_{n}\pq^{n}(z_h)+\kappa_{1}\pq'(z_h)+\ldots+\kappa_{n-1}\pq^{(n-1)}(z_h)
\end{align*}
with $h=1,\ldots,2n-1$.

\begin{theorem}
The equation
        \begin{align*}
        \begin{vmatrix}
1 & \pq(z_1) &  \cdots & \pq^n (z_1) & \pq'(z_1) & \cdots & \pq^{(n-1)}(z_1)  \\
 \vdots& \vdots & \cdots & \vdots \\
1 & \pq(z_{2n-1}) &  \cdots & \pq^n (z_{2n-1}) & \pq'(z_{2n-1}) & \cdots & \pq^{(n-1)}(z_{2n-1})  \\
1 & \pq(z+\rho) &  \cdots & \pq^n (z+\rho) & \pq'(z+\rho) & \cdots & \pq^{(n-1)}(z+\rho).  \\
\end{vmatrix}=0
\end{align*}
holds where $\rho=\rho_{\pq,n}$ is a certain half period.
\begin{proof}
    Let us consider the function
    \begin{align*}
        \psi_{\pq}(u)=        -1+\lambda_{1}\pq(u)+\ldots+\lambda_{n}\pq^{n}(u)+\kappa_{1}\pq'(u)+\ldots+\kappa_{n-1}\pq^{(n-1)}(u).
    \end{align*}
    It is an elliptic function in the class $P_{\pq}$ of order $2n$ for generic values of $z_h$. Assuming $z_h$ are distinct, we have $2n-1$ zeros already. If $w$ denotes the last zero, we have $w+z_1+\ldots+z_{2n-1}\equiv nS_{\pq}$ modulo periods, where $S_{\pq}$ is the half-period given by the sum of the poles. Thus $w\equiv z+\rho$ modulo periods with $\rho=nS_{\pq}$ and $\psi_{\pq}(z+\rho)=0$. Putting these equations together in the matrix form gives us 
    \begin{align*}
        \begin{pmatrix}
1 & \pq(z_1) &  \cdots & \pq^n (z_1) & \pq'(z_1) & \cdots & \pq^{(n-1)}(z_1)  \\
 \vdots& \vdots & \cdots & \vdots \\
1 & \pq(z_{2n-1}) &  \cdots & \pq^n (z_{2n-1}) & \pq'(z_{2n-1}) & \cdots & \pq^{(n-1)}(z_{2n-1})  \\
1 & \pq(z+\rho) &  \cdots & \pq^n (z+\rho) & \pq'(z+\rho) & \cdots & \pq^{(n-1)}(z+\rho).  \\
\end{pmatrix}\begin{pmatrix}
-1 \\
 \lambda_{1}\\
\vdots\\
\kappa_{n-1}
\end{pmatrix}=\begin{pmatrix}
\psi_{\pq}(z_1) \\
\vdots\\
\psi_{\pq}(z_{2n-1})\\
\psi_{\pq}(z+\rho)
\end{pmatrix}=0.
    \end{align*}
    Since the vector on the left-hand side is non-zero, the matrix on the left-hand side is singular and has a vanishing determinant. This concludes the proof.

\end{proof}

\end{theorem}

\subsection{An Explicit Addition Theorem}
In this section we derive numerous explicit formulas for the value $\pq(z+\rho)=\pm \pq(z)$, analogous to that of \cite{Gürel}. We first need the following lemma.
\begin{lemma}
    For all $j\in \mathbb{N}$, the functions $\pq^{(2j)}$ and $\pq^{(2j+1)}/\pq'$ are polynomials in $\pq$ with degrees $2j+1$ and $2j$ respectively.
\end{lemma}
\begin{proof}
    The claim follows immediately after induction and the equation
    \begin{align*}
        \left(\pq'\right)^2=A_{\pq}\pq^4+B_{\pq}\pq^2+C_{\pq}.
    \end{align*}
\end{proof}
It can be seen from the lemma that the function defined as
\begin{align*}
    \varphi(u)=\left( -1+\sum_{j=1}^{n}\lambda_j \pq^j(u)+\sum_{\substack{j=1\\ 2|j}}^{n-1}\kappa_j\pq^{(j)}(u) \right)^2-\left( \sum_{\substack{j=1\\ 2\nmid j}}^{n-1}\kappa_j\pq^{(j)}(u) \right)^2
\end{align*}
is a polynomial in $\pq$ of degree at most $2n$ and thus can be expanded as
\begin{align*}
    \varphi(u)=\sum_{j=0}^{2n}c_j\pq^{j}(u).
\end{align*}
Here the constants $c_j$ can obviously be computed in terms of $\lambda_j,\kappa_j,A_{\pq},B_{\pq},C_{\pq}$.
\begin{theorem}
    We have
    \begin{align*}
        S_j\pq(z_1,\ldots,z_{2n-1},z+\rho)=(-1)^j\frac{c_{2n-j}}{c_{2n}}
    \end{align*}
    or
    \begin{align*}
        \pq(z+\rho)=\frac{1}{S_{j-1}\pq(z_1,\ldots,z_{2n-1})}\left( (-1)^j\frac{c_{2n-j}}{c_{2n}}-S_{j}\pq(z_1,\ldots,z_{2n-1}) \right)
    \end{align*}
    for $j\neq 0$.
\end{theorem}
\begin{proof}
    Notice that $\varphi(z_h)=\varphi(z+\rho)=0$ for all $h=1,\ldots,2n-1$. For generic values of $z_h$, the values $\pq(z_h),\pq(z+\rho)$ are distinct. Therefore we have the factorization
    \begin{align*}
        \varphi(u)&=c_{2n}\left(\pq(u)- \pq(z+\rho) \right)\prod_{j=1}^{2n-1}\left(\pq(u)- \pq(z_j)\right)\\
&=\sum_{j=0}^{2n}c_j\pq^{j}(u)
    \end{align*}
    and in particular $c_{2n}\neq 0$. Comparing the coefficients of $\pq^{2n-j}$, we obtain
    \begin{align*}
        (-1)^j c_{2n}S_j\pq(z_1,\ldots,z_{2n-1},z+\rho)=c_{2n-j}.
    \end{align*}
    The proof is completed upon noting the identity
    \begin{align*}
        S_j\pq(z_1,\ldots,z_{2n-1},z+\rho)=\pq(z+\rho)S_{j-1}\pq(z_1,\ldots,z_{2n-1})+S_{j}\pq(z_1,\ldots,z_{2n-1}).
    \end{align*}
\end{proof}

\begin{example}The most applicable case is $n=2$ with the system
\begin{align*}
    1=\lambda_1\pq(z_h)+\lambda_2\pq^2(z_h)+\kappa_1\pq'(z_h)
\end{align*}
valid for $h=1,2,3$. We obtain the function 
\begin{align*}
        \varphi(u)=\left( -1+\lambda_1\pq(u)+\lambda_2\pq^2(u) \right)^2-\kappa_1^2\left( \pq' \right)^2(u)=\sum_{j=0}^4 c_j \pq^j(u)
\end{align*}
and the constants
\begin{align*}
    c_4&=\lambda_2^2-\kappa_1^2A_{\pq},\\
c_3&=2\lambda_1\lambda_2,\\
c_2&=\lambda_1^2-2\lambda_2-\kappa_1^2B_{\pq},\\
c_1&=-2\lambda_1,\\
c_0&=1-\kappa_1^2C_{\pq}.
\end{align*}
Furthermore, we get
\begin{align*}
    \lambda_1&=\frac{\sum_{cyc}\pq'(z_1)\left(\pq^2(z_3)-\pq^2(z_2) \right)}{\sum_{cyc}\pq'(z_1)\left( \pq(z_2)\pq^2(z_3)-\pq(z_3)\pq^2(z_2) \right) },\\
\lambda_2&=\frac{\sum_{cyc}\pq'(z_1)\left(\pq(z_3)-\pq(z_2) \right)}{\sum_{cyc}\pq'(z_1)\left( \pq(z_2)\pq^2(z_3)-\pq(z_3)\pq^2(z_2) \right) },\\
\kappa_1&=\frac{(\pq(z_2)-\pq(z_1))(\pq(z_3)-\pq(z_1))(\pq(z_3)-\pq(z_2))}{\sum_{cyc}\pq'(z_1)\left( \pq(z_2)\pq^2(z_3)-\pq(z_3)\pq^2(z_2) \right) }
\end{align*}
where the summations are to be performed cyclically. Our formulas read
\begin{align*}
\pq(z)&=-\frac{2\lambda_1\lambda_2}{\lambda_2^2-\kappa_1^2A_{\pq}}-\pq(z_1)-\pq(z_2)-\pq(z_3)\\
&= \frac{\lambda_1^2-2\lambda_2-\kappa_1^2B_{\pq}}{\left(\lambda_2^2-\kappa_1^2A_{\pq}\right)\left(\pq(z_1)+\pq(z_2)+\pq(z_3)\right)}-\frac{\pq(z_1)\pq(z_2)+\pq(z_2)\pq(z_3)+\pq(z_3)\pq(z_1)}{\pq(z_1)+\pq(z_2)+\pq(z_3)}\\
&=\frac{1}{\pq(z_1)\pq(z_2)+\pq(z_2)\pq(z_3)+\pq(z_3)\pq(z_1)}\left( \frac{2\lambda_1}{\lambda_1^2-2\lambda_2-\kappa_1^2B_{\pq}}-\pq(z_1)\pq(z_2)\pq(z_3) \right)\\
&=\frac{1-\kappa_1^2C_{\pq}}{\left(\lambda_2^2-\kappa_1^2A_{\pq}\right)\pq(z_1)\pq(z_2)\pq(z_3)}
\end{align*}
as $\rho$ is a period for $n=2$. We obtain the triplication formula by taking the limit $z_h\to u$. A direct calculation yields
\begin{align*}
    \lim_{z_h\to u}\lambda_1&=\frac{2(\pq \pq'''\pq'-\pq'^2\pq''-\pq\pq''^2)}{\pq'''\pq^2\pq'-\pq^2\pq''^2-2\pq\pq'^2\pq''+2\pq'^4},\\
\lim_{z_h\to u}\lambda_2&=\frac{\pq''^2-\pq'\pq'''}{\pq'''\pq^2\pq'-\pq^2\pq''^2-2\pq\pq'^2\pq''+2\pq'^4},\\
\lim_{z_h\to u}\kappa_1&=\frac{2\pq'^3}{\pq'''\pq^2\pq'-\pq^2\pq''^2-2\pq\pq'^2\pq''+2\pq'^4}
\end{align*}
where we have suppressed the argument $u$. We can use the equations
\begin{align*}
    \pq'^2&=A_{\pq}\pq^4+B_{\pq}\pq^2+C_{\pq},\\
\pq''&=2A_{\pq}\pq^3+B_{\pq}\pq,\\
\pq'''&=\pq'\left( 6A_{\pq}\pq^2 +B_{\pq}\right)
\end{align*}
to write $\lambda_1,\lambda_2,\kappa_1$ entirely in terms of $\pq$. We have
\begin{align*}
    \lim_{z_h\to u}\lambda_1&=\frac{2(4A_{\pq}C_{\pq}-B_{\pq}^2)\pq^3}{A_{\pq}B_{\pq}\pq^6+6A_{\pq}C_{\pq}\pq^4+3B_{\pq}C_{\pq}\pq^2+2C_{\pq}^2},\\
\lim_{z_h\to u}\lambda_2&=-\frac{2A_{\pq}^2\pq^6+3A_{\pq}B_{\pq}\pq^4+6A_{\pq}C_{\pq}\pq^2+B_{\pq}C_{\pq}}{A_{\pq}B_{\pq}\pq^6+6A_{\pq}C_{\pq}\pq^4+3B_{\pq}C_{\pq}\pq^2+2C_{\pq}^2},\\
\lim_{z_h\to u}\kappa_1&=\frac{2\pq'(A_{\pq}\pq^4+B_{\pq}\pq^2+C_{\pq})}{A_{\pq}B_{\pq}\pq^6+6A_{\pq}C_{\pq}\pq^4+3B_{\pq}C_{\pq}\pq^2+2C_{\pq}^2}.
\end{align*}
For example, one obtains the Cayley's triplication formula for $\sn (3u)$ in the case of $\pq=\sn$. We can also choose a fixed value for $z_3$ as a period or a suitable half-period to obtain the classical 2-term addition theorems alongside the duplication formulas upon limiting.
\end{example}

\begin{example}In the case of $n=3$, we get the system
\begin{align*}
    1=\lambda_1\pq(z_h)+\lambda_2\pq^2(z_h)+\lambda_3\pq^3(z_h)+\kappa_1\pq'(z_h)+\kappa_2\pq''(z_h)
\end{align*}
valid for $h=1,2,3,4,5$. We obtain the function 
\begin{align*}
        \varphi(u)=\left( -1+\lambda_1\pq(u)+\lambda_2\pq^2(u)+\lambda_3\pq^3(u)+\kappa_2\pq''(u) \right)^2-\kappa_1^2\left( \pq' \right)^2(u)=\sum_{j=0}^6 c_j \pq^j(u)
\end{align*}
and the constants
\begin{align*}
    c_6&=4A_{\pq}^2\kappa_2^2+4A_{\pq}\kappa_2\lambda_3+\lambda_3^2,\\
c_5&=4A_{\pq}\kappa_2\lambda_2+2\lambda_2\lambda_3,\\
c_4&=4A_{\pq}B_{\pq}\kappa_2^2+4A_{\pq}\kappa_2\lambda_1-A_{\pq}\kappa_1^2+2B_{\pq}\kappa_2\lambda_3+\lambda_2^2+2\lambda_1\lambda_3,\\
c_3&=-4A_{\pq}\kappa_2+2B_{\pq}\kappa_2\lambda_2+2\lambda_1\lambda_2-2\lambda_3,\\
c_2&=B_{\pq}^2\kappa_2^2+2B_{\pq}\kappa_2\lambda_1-B_{\pq}\kappa_1^2+\lambda_1^2-2\lambda_2,\\
c_1&=-2B_{\pq}\kappa_2-2\lambda_1,\\
c_0&=1-C_{\pq}\kappa_1^2.
\end{align*}
Our formulas read
\begin{align*}
    \pq(z+3 S_{\pq})&=-\frac{A_{\pq}\kappa_2\lambda_2+2\lambda_2\lambda_3}{4A_{\pq}^2\kappa_2^2+4A_{\pq}\kappa_2\lambda_3+\lambda_3^2}-\pq(z_1)-\ldots-\pq(z_5) \\
    &= \frac{4A_{\pq}B_{\pq}\kappa_2^2+4A_{\pq}\kappa_2\lambda_1-A_{\pq}\kappa_1^2+2B_{\pq}\kappa_2\lambda_3+\lambda_2^2+2\lambda_1\lambda_3}{\left(\pq(z_1)+\ldots+\pq(z_5)\right)\left(4A_{\pq}^2\kappa_2^2+4A_{\pq}\kappa_2\lambda_3+\lambda_3^2\right)}-\frac{S_{2}\pq(z_1,\ldots,z_{5})}{\pq(z_1)+\ldots+\pq(z_5)} \\
    &=\frac{1}{S_{2}\pq(z_1,\ldots,z_{5})}\left( -\frac{-4A_{\pq}\kappa_2+2B_{\pq}\kappa_2\lambda_2+2\lambda_1\lambda_2-2\lambda_3}{4A_{\pq}^2\kappa_2^2+4A_{\pq}\kappa_2\lambda_3+\lambda_3^2}-S_{3}\pq(z_1,\ldots,z_{5}) \right)\\
    &=\frac{1}{S_{3}\pq(z_1,\ldots,z_{5})}\left( \frac{B_{\pq}^2\kappa_2^2+2B_{\pq}\kappa_2\lambda_1-B_{\pq}\kappa_1^2+\lambda_1^2-2\lambda_2}{4A_{\pq}^2\kappa_2^2+4A_{\pq}\kappa_2\lambda_3+\lambda_3^2}-S_{4}\pq(z_1,\ldots,z_{5}) \right)\\
    &=\frac{1}{S_{4}\pq(z_1,\ldots,z_{5})}\left( \frac{2B_{\pq}\kappa_2+2\lambda_1}{4A_{\pq}^2\kappa_2^2+4A_{\pq}\kappa_2\lambda_3+\lambda_3^2}-\pq(z_1)\ldots\pq(z_5) \right)\\
    &=\frac{1-C_{\pq}\kappa_1^2}{\pq(z_1)\ldots\pq(z_5) \left(4A_{\pq}^2\kappa_2^2+4A_{\pq}\kappa_2\lambda_3+\lambda_3^2\right)}.
\end{align*}
The limiting process $z_h\to u$ yields 5-multiplication formulas. We may again choose special values for $z_5$ to obtain 4-term addition and 4-multiplication formulas.
\end{example}
In the case of $n=4$, we get the system
\begin{align*}
    1=\lambda_1\pq(z_h)+\lambda_2\pq^2(z_h)+\lambda_3\pq^3(z_h)+\lambda_4\pq^4(z_h)+\kappa_1\pq'(z_h)+\kappa_2\pq''(z_h)+\kappa_3\pq'''(z_h)
\end{align*}
valid for $h=1,2,3,4,5,6,7$. We obtain the function 
\begin{align*}
        \varphi(u)&=\left( -1+\lambda_1\pq(u)+\lambda_2\pq^2(u)+\lambda_3\pq^3(u)+\lambda_4\pq^4(z_h)+\kappa_2\pq''(u) \right)^2-\left(\kappa_1\pq'(u)+\kappa_3\pq'''(z_h)\right)^2\\
        &=\sum_{j=0}^8 c_j \pq^j(u)
\end{align*}
and the constants
\begin{align*}
    c_8&=\lambda_4^2-36A_{\pq}^3\kappa_3^2,\\
    c_7&=4A_{\pq}\lambda_4\kappa_2+2\lambda_3\lambda_4,\\
    c_6&=4A_{\pq}^2\kappa_2^2-12A_{\pq}^3\kappa_1\kappa_3-48A_{\pq}^2B_{\pq}\kappa_3^2+4A_{\pq}\kappa_2\lambda_3+\lambda_3^2+2\lambda_2\lambda_4,\\
    c_5&=2\lambda_2\lambda_3+2\lambda_1\lambda_4+2A_{\pq}\kappa_2\lambda_2+B_{\pq}\kappa_2\lambda_4,\\
    c_4&=-A_{\pq}\kappa_1^2+4A_{\pq}B_{\pq}\kappa_2^2-14A_{\pq}B_{\pq}\kappa_1\kappa_3-13A_{\pq}B_{\pq}^2 \kappa_3^2-36A_{\pq}^2C_{\pq}\kappa_3^2+\lambda_2^2\\
    &+2\lambda_1\lambda_3+4A_{\pq}\kappa_2\lambda_1+2B_{\pq}\kappa_2\lambda_3-2\lambda_4,\\
    c_3&=2\lambda_1\lambda_2-4A_{\pq}\kappa_2+2B_{\pq}\kappa_2\lambda_2-2\lambda_3,\\
    c_2&=-B_{\pq}\kappa_1^2+B_{\pq}^2\kappa_2^2-2B_{\pq}^2\kappa_1\kappa_3-12A_{\pq}C_{\pq}\kappa_1\kappa_3-B_{\pq}^3\kappa_3^2\\
    &-12A_{\pq}B_{\pq}C_{\pq}\kappa_3^2+2B_{\pq}\kappa_2\lambda_1+\lambda_1^2-2\lambda_2,\\
    c_1&=-2B_{\pq}\kappa_2-2\lambda_1,\\
    c_0&=1-C_{\pq}\kappa_1^2-2B_{\pq}C_{\pq}\kappa_1\kappa_3-B_{\pq}^2C_{\pq}\kappa_3^2.
\end{align*}
Our formulas read
\begin{align*}
    \pq(z)=\frac{1}{S_{j-1}\pq(z_1,\ldots,z_{7})}\left( (-1)^j\frac{c_{8-j}}{c_{8}}-S_{j}\pq(z_1,\ldots,z_{7}) \right)
\end{align*}
for all $j=1\ldots, 8$. The limiting process $z_h\to u$ yields 7-multiplication formulas. We may again choose special values for $z_7$ to obtain 6-term addition and 6-multiplication formulas.

\end{document}